\documentclass[10pt]{amsart}
\usepackage{amsmath,amscd}
\usepackage{amsbsy}
\usepackage{amssymb}
\usepackage{amscd,amsthm}
\usepackage[all,cmtip]{xy}

\newtheorem{thm}{Theorem}\numberwithin{thm}{section}
\newtheorem{lem}[thm]{Lemma}
\newtheorem{prop}[thm]{Proposition}
\newtheorem{cor}[thm]{Corollary}
\newtheorem{exam}[thm]{Example}
\newtheorem{rema}[thm]{Remark}

\newtheorem{con}[thm]{Conjecture}
\newtheorem{que}[thm]{Question}

\newtheorem{defi}[thm]{Definition}

\newtheorem*{thm2}{Theorem}

\begin{document}
\begin{center}
\huge{Some remarks on rigid sheaves, helices and exceptional vector bundles on Fano varieties over arbitrary fields}\\[1cm]
\end{center}
\begin{center}

\large{Sa$\mathrm{\check{s}}$a Novakovi$\mathrm{\acute{c}}$}\\[0,5cm]
{\small March 2018}\\[0,5cm]
\end{center}
{\small \textbf{Abstract}. 
In this paper we study the connection between rigid sheaves and separable exceptional objects on Fano varieties over arbitrary fields. We give criteria for a rigid vector bundle on a Fano variety to be the direct sum of separable exceptional bundles and for a separable exceptional vector bundle to be part of a full separable exceptional collection. 
\begin{center}
\tableofcontents
\end{center}
\section{Introduction}
Beilinson \cite{BE} first constructed full strong exceptional collections on $\mathbb{P}^n_k$ and started the process that people laid hands on concrete derived categories of geometrical significance. His theorem was then applied to the study of moduli spaces of semistable sheaves on $\mathbb{P}^n_k$ (see \cite{OSS} and references therein). Later a considerable understanding for zero-dimensional moduli spaces was obtained. The considered sheaves in this moduli problem were the so-called rigid sheaves on Fano varieties respectively del Pezzo surfaces over algebraically closed fields $k$. Recall that coherent sheaf $\mathcal{F}$ on a Fano variety is called \emph{rigid} if $\mathrm{Ext}^1(\mathcal{F},\mathcal{F})=0$. Kuleshov and Orlov \cite{KO} studied in detail rigid and exceptional vector bundles on del Pezzo surfaces over $\mathbb{C}$ and proved among others that any rigid vector bundle is the direct sum of exceptional ones. In \cite{RU} Rudakov motivated the investigation of rigid sheaves on Fano varieties respectively del Pezzo surfaces over arbitrary ground fields $k$. But in dealing with Fano varieties over arbitrary base fields $k$ it turns out that one should study weak exceptional (or more generally semi- or separable exceptional) instead of exceptional objects (see \cite{NO2}, \cite{NO3}). The only difference in the definition is that for weak exceptional (resp. semi- or separable exceptional) objects $\mathcal{E}$ the ring $\mathrm{End}(\mathcal{E})$ is required to be a division (resp. semisimple or separable) algebra over $k$. And indeed, there are Fano varieties admitting full weak exceptional collections but not full exceptional ones (see \cite{NO1} and \cite{NO2}).

To come back to Rudakov, in \cite{RU} it is conjectured that on del Pezzo surfaces over arbitrary fields $k$ any rigid sheaf is the direct sum of weak exceptional ones. Of course, one can also ask if any rigid sheaf on an arbitrary Fano variety is the direct sum of semi-exceptional (resp. separable exceptional) vector bundles. In the present paper we want to shed light on this question. Recall that a collection $\{\mathcal{E}_1,...,\mathcal{E}_n\}$ of objects in $D^b(X)$ is called \emph{semi-exceptional block} if $\mathrm{Hom}(\mathcal{E}_i,\mathcal{E}_j[l])=0$ for any $i,j$ whenever $l\neq 0$. If furthermore $\mathrm{End}(\bigoplus^n_{i=1}\mathcal{E}_i)$ is the product of matrix algebras over $k$ the collection is called \emph{split semisimple exceptional block} (see Definition 3.9). 
\begin{thm2}[Theorem 6.12]
Let $X$ be a Fano variety over $k$ and $\mathcal{E}$ a rigid vector bundle on $X$. Let $\mathcal{E}\otimes_k k^s=\mathcal{A}_1\oplus...\oplus\mathcal{A}_r$ the Krull--Schmidt decomposition on $X\otimes_k k^s$. Denote by $\mathbb{B}_i=\{\mathcal{B}_{i_{1}},...,\mathcal{B}_{i_{r_i}}\}$ the orbit of $\mathcal{A}_i$ under the action of $\mathrm{Gal}(k^s|k)$. Then $\mathcal{E}$ is the direct sum of separable exceptional vector bundles if and only if all $\mathbb{B}_i$ are split semisimple exceptional blocks over $k^s$.
\end{thm2}
In the special case where $\#\mathbb{B}_i=1$, we obtain the following result.
\begin{thm2}[Theorem 6.14]
Let $X$ be a Fano variety and $\mathcal{E}$ a rigid vector bundle. Let $\mathcal{E}\otimes_k k^s=\mathcal{A}_1\oplus...\oplus\mathcal{A}_r$ the Krull--Schmidt decomposition on $X\otimes_k k^s$. Denote by $\mathbb{B}_i=\{\mathcal{B}_{i_{1}},...,\mathcal{B}_{i_{r_i}}\}$ the orbit of $\mathcal{A}_i$ under the action of $\mathrm{Gal}(k^s|k)$. If $\#\mathbb{B}_i=1$, then $\mathcal{E}$ is the direct sum of weak exceptional sheaves. 
\end{thm2}

Related to Theorem 6.12 is the problem whether, for instance, any separable exceptional sheaf on a del Pezzo surface (or more general on a Fano variety) can be included into a full separable exceptional collection. Analogously to Theorem 6.12, we prove the following criterion for an separable exceptional collection to be part of a full separable exceptional collection:
\begin{thm2}[Theorem 6.13]
Let $X$ be a Fano variety over $k$ and $\mathcal{E}$ a separable exceptional vector bundle. Let $\mathcal{E}\otimes_k k^s=\mathcal{A}_1\oplus...\oplus\mathcal{A}_r$ the Krull--Schmidt decomposition on $X\otimes_k k^s$. Denote by $\mathbb{B}_i=\{\mathcal{B}_{i_{1}},...,\mathcal{B}_{i_{r_i}}\}$ the orbit of $\mathcal{A}_i$ under the action of $\mathrm{Gal}(k^s|k)$. Then $\mathcal{E}$ can be included into a full separable exceptional collection on $X$ if and only if the $\mathbb{B}_i$ can be included into a full exceptional collection on $X\otimes_k k^s$ consisting of Galois invariant split semisimple blocks.
\end{thm2}
In view of the fact that on del Pezzo surfaces over an algebraically closed field any rigid vector bundle is the direct sum of exceptional ones and that furthermore any exceptional collection can be included into a full exceptional collection (see \cite{KO}), the above criteria show what to prove to produce examples which give evidence for Rudakov's conjecture to be true.\\

{\small \textbf{Conventions}. Throughout this work $k$ denotes an arbitrary ground field and $k^s$ and $\bar{k}$ a separable respectively algebraic closure. Furthermore, any locally free sheaf is assumed to be of finite rank and will be called vector bundle.

\section{Fano varieties and del Pezzo surfaces over arbitrary fields}

By a \emph{variety over} $k$ we mean a separated scheme of finite type over $k$. In this paper we will consider only smooth projective and geometrically integral varieties. A \emph{Fano variety} is by definition a variety over $k$ with ample anti-canonical sheaf. A \emph{del Pezzo surface} is a Fano variety of dimension two. 

If $X$ is a del Pezzo surface we write $K_X$ for the class of $\omega_X$ in $\mathrm{Pic}(X)$. The intersection number $d=(K_X,K_X)$ is called the \emph{degree} of $X$. Riemann--Roch theorem and Castelnuovo's criterion show that $X\otimes_k k^s$ is rational. Moreover, $X\otimes_k k^s$ is isomorphic to either $\mathbb{P}^1\times \mathbb{P}^1$, or to the blow up of $\mathbb{P}^2$ at $r\leq 8$ distinct closed points. The precise statement is the content of the following well-known theorem (see \cite{V}, Theorem 1.6).
\begin{thm}
Let $X$ be a del Pezzo surface over a separably closed field $k$ of degree $d$. Then either $X$ is isomorphic to the blow up of $\mathbb{P}^2_k$ at $9-d$ points in general position in $\mathbb{P}^2_k(k)$, or $d=8$ and $X$ is isomorphic to $\mathbb{P}^1_k\times \mathbb{P}^1_k$.
\end{thm}
To give some non-trivial examples of Fano varieties, we recall some basic facts on Brauer--Severi varieties. For details we refer to \cite{AR} and \cite{GS}.

Recall that a finite-dimensional $k$-algebra $A$ is called \emph{central simple} if it is an associative $k$-algebra that has no two-sided ideals other than $0$ and $A$ and if its center equals $k$. If the algebra $A$ is a division algebra it is called \emph{central division algebra}. Note that $A$ is a central simple $k$-algebra if and only if there is a finite field extension $k\subset L$, such that $A\otimes_k L \simeq M_n(L)$. This is also equivalent to $A\otimes_k \bar{k}\simeq M_n(\bar{k})$. An extension $k\subset L$ such that $A\otimes_k L\simeq M_n(L)$ is called splitting field for $A$. 

The \emph{degree} of a central simple algebra $A$ is defined to be $\mathrm{deg}(A):=\sqrt{\mathrm{dim}_k A}$. According to the \emph{Wedderburn Theorem}, for any central simple $k$-algebra $A$ there is an unique integer $n>0$ and a division $k$-algebra $D$ such that $A\simeq M_n(D)$. The division algebra $D$ is unique up to isomorphism and its degree is called the \emph{index} of $A$ and is denoted by $\mathrm{ind}(A)$. 

A \emph{Brauer--Severi variety} of dimension $n$ is a variety $X$ such that $X\otimes_k L\simeq \mathbb{P}^n$ for a finite field extension $k\subset L$. A field extension $k\subset L$ for which $X\otimes_k L\simeq \mathbb{P}^n$ is called \emph{splitting field} of $X$. Clearly, $k^s$ and $\bar{k}$ are splitting fields for any Brauer--Severi variety. In fact, every Brauer--Severi variety always splits over a finite separable field extension of $k$. By embedding the finite separable splitting field into its Galois closure, a Brauer--Severi variety therefore always splits over a finite Galois extension. It follows from descent theory that $X$ is projective, integral and smooth over $k$. Via Galois cohomology, $n$-dimensional Brauer--Severi varieties are in one-to-one correspondence with central simple algebras of degree $n+1$. Note that for a $n$-dimensional Brauer--Severi variety $X$ one has $\omega_X=\mathcal{O}_X(-n-1)$. For details and proofs on all mentioned facts we refer to \cite{AR} and \cite{GS}.

To a central simple $k$-algebra $A$ one can also associate twisted forms of Grassmannians. Let $A$ be of degree $n$ and $1\leq d\leq n$. Consider the subset of $\mathrm{Grass}_k(d\cdot n, A)$ consisting of those subspaces of $A$ that are left ideals $I$ of dimension $d\cdot n$. This subset can be given the structure of a projective variety which turns out to be a generalized Brauer--Severi variety. It is denoted by $\mathrm{BS}(d,A)$. After base change to some splitting field $E$ of $A$ the variety $\mathrm{BS}(d,A)$ becomes isomorphic to $\mathrm{Grass}_E(d,n)$. For $d=1$ the generalized Brauer--Severi variety is the Brauer--Severi variety associated to $A$. Note that $\mathrm{BS}(d,A)$ is a Fano variety. For details and properties of generalized Brauer--Severi varieties see \cite{BL}. 

Note that the finite product $X$ of (generalized) Brauer--Severi varieties is a Fano variety. In particular, if $X$ is two-dimensional, it is a del Pezzo surface.

\section{Tilting objects and exceptional collections}
Let $\mathcal{D}$ be a triangulated category and $\mathcal{C}$ a triangulated subcategory. The subcategory $\mathcal{C}$ is called \emph{thick} if it is closed under isomorphisms and direct summands. For a subset $A$ of objects of $\mathcal{D}$ we denote by $\langle A\rangle$ the smallest full thick subcategory of $\mathcal{D}$ containing the elements of $A$. 
Furthermore, we define $A^{\perp}$ to be the subcategory of $\mathcal{D}$ consisting of all objects $M$ such that $\mathrm{Hom}_{\mathcal{D}}(E[i],M)=0$ for all $i\in \mathbb{Z}$ and all elements $E$ of $A$. We say that $A$ \emph{generates} $\mathcal{D}$ if $A^{\perp}=0$. Now assume $\mathcal{D}$ admits arbitrary direct sums. An object $B$ is called \emph{compact} if $\mathrm{Hom}_{\mathcal{D}}(B,-)$ commutes with direct sums. Denoting by $\mathcal{D}^c$ the subcategory of compact objects we say that $\mathcal{D}$ is \emph{compactly generated} if the objects of $\mathcal{D}^c$ generate $\mathcal{D}$. One has the following important theorem (see \cite{BV}, Theorem 2.1.2).
\begin{thm}
Let $\mathcal{D}$ be a compactly generated triangulated category. Then a set of objects $A\subset \mathcal{D}^c$ generates $\mathcal{D}$ if and only if $\langle A\rangle=\mathcal{D}^c$.  
\end{thm}
For a smooth projective variety $X$ over $k$, we denote by $D(\mathrm{Qcoh}(X))$ the derived category of quasicoherent sheaves on $X$. The bounded derived category of coherent sheaves is denoted by $D^b(X)$. Note that $D(\mathrm{Qcoh}(X))$ is compactly generated with compact objects being all of $D^b(X)$. For details on generating see \cite{BV}.
\begin{defi}
\textnormal{Let $k$ be a field and $X$ a smooth projective variety over $k$. An object $\mathcal{T}\in D(\mathrm{Qcoh}(X))$ is called \emph{tilting object} on $X$ 
if the following hold:
\begin{itemize}
      \item[\bf (i)] Ext vanishing: $\mathrm{Hom}(\mathcal{T},\mathcal{T}[i])=0$ for $i\neq0$.
     \item[\bf (ii)] Generation: If $\mathcal{N}\in D(\mathrm{Qcoh}(X))$ satisfies $\mathbb{R}\mathrm{Hom}(\mathcal{T},\mathcal{N})=0$, then $\mathcal{N}=0$.
		\item[\bf (iii)] Compactness: $\mathrm{Hom}(\mathcal{T},-)$ commutes with direct sums.
		\end{itemize}}
\end{defi}
Below we state the well-known tilting correspondence which is a direct application of a more general result on triangulated categories (see \cite{KE}, Theorem 8.5). We denote by $\mathrm{Mod}(A)$ the category of right $A$-modules and by $D^b(A)$ the bounded derived category of finitely generated right $A$-modules. Furthermore, $\mathrm{perf}(A)\subset D(\mathrm{Mod}(A)) $ denotes the full triangulated subcategory of perfect complexes, those quasi-isomorphic to a bounded complexes of finitely generated projective right $A$-modules.
\begin{thm}
Let $X$ be a smooth projective variety over $k$. Suppose we are given a tilting object $\mathcal{T}$ on $X$ 
and let $A=\mathrm{End}(\mathcal{T})$. Then the following hold:
\begin{itemize}
      \item[\bf (i)] The functor $\mathbb{R}\mathrm{Hom}(\mathcal{T},-)\colon D(\mathrm{Qcoh}(X))\rightarrow D(\mathrm{Mod}(A))$ is an equivalence. 
      \item[\bf (ii)] If $\mathcal{T}\in D^b(X)$, this equivalence restricts to an equivalence $D^b(X)\stackrel{\sim}\rightarrow D^b(A)$.
\end{itemize} 
\end{thm}
\begin{defi}
\textnormal{Let $A$ be a division algebra over $k$, not necessarily central. An object $\mathcal{E}\in D^b(X)$ is called \emph{weak exceptional} if $\mathrm{End}(\mathcal{E})=A$ and $\mathrm{Hom}(\mathcal{E},\mathcal{E}[r])=0$ for $r\neq 0$. If $A=k$ the object is called \emph{exceptional}. If $A$ is a separable $k$-algebra, the object $\mathcal{E}$ is called \emph{separable exceptional}.} 
\end{defi}
\begin{defi}
\textnormal{A totally ordered set $\{\mathcal{E}_1,...,\mathcal{E}_n\}$ of weak exceptional (resp. separable exceptional) objects on $X$ is called a \emph{weak exceptional collection} (resp. \emph{separable exceptional collection}) if $\mathrm{Hom}(\mathcal{E}_i,\mathcal{E}_j[r])=0$ for all integers $r$ whenever $i>j$. A weak exceptional (resp. separable exceptional) collection is \emph{full} if $\langle\{\mathcal{E}_1,...,\mathcal{E}_n\}\rangle=D^b(X)$ and \emph{strong} if $\mathrm{Hom}(\mathcal{E}_i,\mathcal{E}_j[r])=0$ whenever $r\neq 0$. If the set $\{\mathcal{E}_1,...,\mathcal{E}_n\}$ consists of exceptional objects it is called \emph{exceptional collection}.}
\end{defi}
Notice that the direct sum of objects forming a full strong weak exceptional (resp. separable exceptional) collection is a tilting object in the sense of Definition 3.2. 
\begin{rema}
\textnormal{If the ring $A$ in Definition 3.4 is required to be a semisimple algebra, the object is also called \emph{semi-exceptional object} in the literature (see \cite{OR}). Consequently, one can also define (full) semi-exceptional collections.}
\end{rema}
\begin{exam}
\textnormal{Let $\mathbb{P}^n$ be the projective space and consider the ordered collection of invertible sheaves $\{\mathcal{O}_{\mathbb{P}^n}, \mathcal{O}_{\mathbb{P}^n}(1),...,\mathcal{O}_{\mathbb{P}^n}(n)\}$. In \cite{BE} Beilinson showed that this is a full strong exceptional collection. }
\end{exam}
\begin{exam}
\textnormal{Let $X=\mathbb{P}^1\times\mathbb{P}^1$. Then $\{\mathcal{O}_X,\mathcal{O}_X(1,0), \mathcal{O}_X(0,1), \mathcal{O}_X(1,1)\}$ is a full strong exceptional collection on $X$. Here we write $\mathcal{O}_X(i,j)$ for $\mathcal{O}(i)\boxtimes\mathcal{O}(j)$.}
\end{exam}
\begin{defi}
\textnormal{Let $X$ be a smooth projective variety over $k$. A collection $\{\mathcal{E}_1,...,\mathcal{E}_n\}$ of objects in $D^b(X)$ is called \emph{semi-exceptional block} if $\mathrm{Hom}(\mathcal{E}_i,\mathcal{E}_j[l])=0$ for any $i,j$ whenever $l\neq 0$. If furthermore $\mathrm{End}(\bigoplus^n_{i=1}\mathcal{E}_i)$ is the product of matrix algebras over $k$, the collection is called \emph{split semisimple exceptional block}.}
\end{defi}

A generalization of the notion of a full weak exceptional collection is that of a semiorthogonal decomposition of $D^b(X)$. Recall that a full triangulated subcategory $\mathcal{D}$ of $D^b(X)$ is called \emph{admissible} if the inclusion $\mathcal{D}\hookrightarrow D^b(X)$ has a left and right adjoint functor. 
\begin{defi}
\textnormal{Let $X$ be a smooth projective variety over $k$. A sequence $\mathcal{D}_1,...,\mathcal{D}_n$ of full triangulated subcategories of $D^b(X)$ is called \emph{semiorthogonal} if all $\mathcal{D}_i\subset D^b(X)$ are admissible and $\mathcal{D}_j\subset \mathcal{D}_i^{\perp}=\{\mathcal{F}\in D^b(X)\mid \mathrm{Hom}(\mathcal{G},\mathcal{F})=0$, $\forall$ $ \mathcal{G}\in\mathcal{D}_i\}$ for $i>j$. Such a sequence defines a \emph{semiorthogonal decomposition} of $D^b(X)$ if the smallest full thick subcategory containing all $\mathcal{D}_i$ equals $D^b(X)$.}
\end{defi}
For a semiorthogonal decomposition we write $D^b(X)=\langle \mathcal{D}_1,...,\mathcal{D}_n\rangle$.
\begin{rema}
\textnormal{Let $\mathcal{E}_1,...,\mathcal{E}_n$ be a full weak exceptional collection on $X$. It is easy to verify that by setting $\mathcal{D}_i=\langle\mathcal{E}_i\rangle$ one gets a semiorthogonal decomposition $D^b(X)=\langle \mathcal{D}_1,...,\mathcal{D}_n\rangle$.}
\end{rema}
For a wonderful and comprehensive overview of the theory on semiorthogonal decompositions and its relevance in algebraic geometry we refer to \cite{KU}.

\section{w-helices on Fano varieties}
We slightly extend the definition of a helix which is given in \cite{BS}.

\begin{defi}
\textnormal{As sequence of objects $\mathbb{H}=(\mathcal{V}_i)_{i\in\mathbb{Z}}$ on a Fano variety $X$ is called a \emph{w-helix of type $(n,d)$} if there are positive integers $n,d$ with $d\geq 2$ such that
\begin{itemize}
      \item[\bf (i)] for each $l\in\mathbb{Z}$ the corresponding \emph{thread} $(\mathcal{V}_{l+1},...,\mathcal{V}_{l+n})$ is a full w-exceptional collection,
     \item[\bf (ii)] for each $l\in\mathbb{Z}$ one has $\mathcal{V}_{l-n}=(\mathcal{V}_l\otimes \omega_X)[\mathrm{dim}(X)+1-d]$.
\end{itemize}}
\end{defi}
A w-helix $\mathbb{H}=(\mathcal{V}_i)_{i\in\mathbb{Z}}$ of type $(n,d)$ is said to be \emph{geometric} if $\mathrm{Hom}(\mathcal{V}_i,\mathcal{V}_j[r])=0$ for all $i<j$ unless $r=0$. Moreover, it is called \emph{strong} if each thread is a full strong w-exceptional collection.
\begin{rema}
\textnormal{If the corresponding thread in Definition 4.1 is a full exceptional collection, the w-helix is called helix in the literature (see \cite{BRI}, \cite{BS} or \cite{KO}). Instead of using full weak exceptional collections in the threads one can also use full semi-exceptional collections.}
\end{rema}
\begin{exam}
\textnormal{Take $X=\mathbb{P}^{d-1}$. Note that $\omega_X=\mathcal{O}_X(-d)$. Then $\mathbb{H}=(\mathcal{O}_X(i))_{i\in \mathbb{Z}}$ is a geometric helix of type $(d,d)$.} 
\end{exam}
\begin{exam}
\textnormal{Take $X=\mathbb{P}^1\times\mathbb{P}^1$. The canonical sheaf is $\omega_X=\mathcal{O}_X(-2,-2)$. Then $\mathbb{H}=(...,\mathcal{O}_X,\mathcal{O}_X(1,0), \mathcal{O}_X(0,1), \mathcal{O}_X(1,1),\mathcal{O}_X(2,2),...)$ is a geometric helix of type $(4,3)$.}
\end{exam}
One can give a geometric interpretation of a w-helix via the so called \emph{rolled-up w-helix algebras}. For this, let $\mathbb{H}=(\mathcal{V}_i)_{i\in\mathbb{Z}}$ be a w-helix of type $(n,d)$ and define the \emph{w-helix algebra} as 
\begin{eqnarray*}
A(\mathbb{H})=\bigoplus_{l\geq 0}\prod_{j-i=l}{\mathrm{Hom}(\mathcal{V}_i,\mathcal{V}_j)}.
\end{eqnarray*}
This is a graded algebra which has a $\mathbb{Z}$-action induced by the Serre functor
\begin{eqnarray*}
(-\otimes\omega_X)[1-d]\colon \mathrm{Hom}(\mathcal{V}_i,\mathcal{V}_j)\longrightarrow \mathrm{Hom}(\mathcal{V}_{i-n},\mathcal{V}_{j-n}). 
\end{eqnarray*}
The \emph{rolled-up w-helix algebra} $B(\mathbb{H})$ is defined to be the subalgebra of $A(\mathbb{H})$ of invariant elements. Obviously, the algebra $B(\mathbb{H})$ is graded, too. To both algebras $A(\mathbb{H})$ and $B(\mathbb{H})$ one can associate a quiver. The quiver underlying $A(\mathbb{H})$ has vertices labeled by the elements of $\mathbb{Z}$ and $a_{i,j}$ arrows connecting vertex $i$ with vertex $j$. Here $a_{i,j}$ denotes the dimension of the cokernel map
\begin{eqnarray*}
\bigoplus_{i<l<j}\mathrm{Hom}(\mathcal{V}_i,\mathcal{V}_l)\otimes \mathrm{Hom}(\mathcal{V}_l,\mathcal{V}_j)\longrightarrow \mathrm{Hom}(\mathcal{V},\mathcal{V}_j).
\end{eqnarray*} 
The quiver underlying $B(\mathbb{H})$ has vertices corresponding to elements of $\mathbb{Z}/n\mathbb{Z}$ and \begin{eqnarray*}
n_{ij}=\sum_{p\in \mathbb{Z}}{a_{i,j+pn}}
\end{eqnarray*}
arrows from vertex $i$ to vertex $j$. The geometric interpretation of the w-helix $\mathbb{H}$ in terms of the rolled-up w-helix algebra $B(\mathbb{H})$ is stated in the corollary below. We first fix some notation. 
For a smooth projective variety $X$ over $k$ let $\mathbb{A}(\mathcal{E}):=\mathcal{S}pec(S^{\bullet}(\mathcal{E}))$, where $S^{\bullet}(\mathcal{E})$ is the symmetric algebra of the vector bundle $\mathcal{E}$ on $X$. The associated structure morphism is $\pi\colon \mathbb{A}(\mathcal{E})\rightarrow X$. Recall the following theorem (see \cite{NO1}, Theorem 5.1). 
\begin{thm}
Let $X$ be a smooth projective variety over $k$ and $\mathcal{E}$ a vector bundle. Suppose $\mathcal{T}$ is a tilting bundle on $X$. If $H^i(X,\mathcal{T}^{\vee}\otimes \mathcal{T}\otimes S^l(\mathcal{E}))=0$ for all $i\neq 0$ and all $l>0$, then $\pi^*\mathcal{T}$ is a tilting bundle on $\mathbb{A}(\mathcal{E})$.
\end{thm} 
Now let $X$ be a Fano variety and set $Y:=\mathbb{A}(\omega^{\vee}_X)$. Furthermore, let $n=\mathrm{rk}(K_0(X))$ and $d=\mathrm{dim}(Y)$. Now Theorem 4.5 has the following consequence.
\begin{cor}
Let $X$ be a Fano variety and $B(\mathbb{H})$ the rolled up w-helix algebra of a given geometric w-helix $\mathbb{H}=(\mathcal{V}_i)_{i\in\mathbb{Z}}$ of type $(n,d)$. Then $B(\mathbb{H})$ is a graded algebra and for a given thread $\mathcal{V}_1,...,\mathcal{V}_n$ there is an equivalence
\begin{eqnarray*}
\mathbb{R}\mathrm{Hom}(\pi^*(\bigoplus_i\mathcal{V}_i),-)\colon D^b(Y)\longrightarrow D^b(B(\mathbb{H})).
\end{eqnarray*} 
\end{cor}
\begin{proof}
According to Theorem 4.5 we only have to verify that $H^i(X,\mathcal{T}^{\vee}\otimes \mathcal{T}\otimes S^l(\mathcal{\omega^{\vee}_X}))=0$ for all $i\neq 0$ and all $l>0$, where $\mathcal{T}=\bigoplus^n_{j=1}\mathcal{V}_j$. Note that $S^l(\omega^{\vee}_X)=\omega^{-l}_X$ so that the vanishing of the desired cohomology follows from the fact that $\mathbb{H}=(\mathcal{V}_i)_{i\in\mathbb{Z}}$ is a geometric w-helix of type $(n,d)$. It is easy to see that $\mathrm{End}(\mathcal{T})\simeq B(\mathbb{H})$. The rest is (ii) of Theorem 3.3. 
\end{proof}

\section{Examples of w-helices on Fano-varieties}
In this section we consider the Fano varieties from Example 2.2 and provide two examples of w-helices. Recall from \cite{NO} the following definition.
\begin{defi}
\textnormal{Let $X$ be a variety over $k$. A vector bundle $\mathcal{E}$ on $X$ is called \emph{absolutely split} if it splits as a direct sum of invertible sheaves on $X\otimes_k \bar{k}$. For an absolutely split vector bundle we shortly write \emph{AS-bundle}.}
\end{defi}
In \cite{NO} we classify all $AS$-bundles on proper $k$-schemes. Among others, we study in detail the $AS$-bundles on Brauer--Severi varieties. So in Section 6 of \textit{loc.cit}. it is proved that on an arbitrary Brauer--Severi variety $X$ the indecomposable $AS$-bundles are vector bundles $\mathcal{W}_i$, $i\in \mathbb{Z}$, satisfying $\mathcal{W}_i\otimes_k \bar{k}\simeq \mathcal{O}(i)^{\oplus \mathrm{ind}(A^{\otimes i})}$, where $A$ is the central simple algebra corresponding to $X$. These $\mathcal{W}_i$ are unique up to isomorphism and one has $\mathcal{W}_0\simeq \mathcal{O}_X$. Furthermore, the vector bundles $\mathcal{W}_i$ satisfy a symmetry and periodicity relation which are stated in the following lemma.
\begin{lem}
Let $X$ be a $n$-dimensional Brauer--Severi variety of period $p$. Then the following hold:
\begin{itemize}
      \item[\bf (i)] $\mathcal{W}^{\vee}_i\simeq \mathcal{W}_{-i}$,
      \item[\bf (ii)] $\mathcal{W}_{i+rp}\simeq \mathcal{W}_i\otimes \mathcal{O}_X(rp)$. 
\end{itemize} 
\end{lem}
\begin{proof}
This follows from \cite{NO}, Proposition 5.3.
\end{proof}
\begin{prop}
Let $X$ be a $n$-dimensional Brauer--Severi variety and $\mathcal{W}_i$ the indecomposable $AS$-bundles from above. Then $\mathbb{H}=(\mathcal{W}_i)_{i\in\mathbb{Z}}$ is a geometric w-helix of type $(n+1,n+1)$.
\end{prop}
\begin{proof}
As $\mathcal{W}_i\otimes_k \bar{k}\simeq \mathcal{O}(i)^{\oplus \mathrm{ind}(A^{\otimes i})}$ we get $\mathrm{End}(\mathcal{W}_i)\otimes_k \bar{k}\simeq M_{\mathrm{ind}(A^{\otimes i})}(\bar{k})$ and hence $\mathrm{End}(\mathcal{W}_i)$ is a central simple $k$-algebra. In fact $\mathrm{End}(\mathcal{W}_i)$ is a central division algebra since $\mathcal{W}_i$ is indecomposable by construction. That $\mathbb{H}=(\mathcal{W}_i)_{i\in\mathbb{Z}}$ is a geometric w-helix of type $(n+1,n+1)$ now follows from base change to some splitting field and (ii) of Lemma 5.2 as the period $p$ divides $n+1$.
\end{proof}
Let $X$ be the del Pezzo surface $C_1\times C_2$, where $C_1$ and $C_2$ are Brauer--Severi curves corresponding to quaternion algebras $(a,b)$ and $(c,d)$ with $a\neq c$ and $b\neq d$. The indecomposable $AS$-bundles on $C_1$ and $C_2$ are denoted by $\mathcal{V}_i$ and $\mathcal{W}_j$ respectively.
\begin{prop}
Let $X$ be as above. Then $\mathbb{H}=(...,\mathcal{O}_X,\mathcal{V}_1\boxtimes\mathcal{O}_{C_2},\mathcal{O}_{C_1}\boxtimes\mathcal{W}_1,\mathcal{V}_1\boxtimes\mathcal{W}_1,\mathcal{O}_{C_1}(2)\boxtimes\mathcal{O}_{C_2}(2),...)$ is a geometric w-helix of type $(4,3)$.
\end{prop}
\begin{proof}
Applying the K\"unneth formula (see \cite{HUY}, p.86), we find $\mathrm{End}(\mathcal{V}_i\boxtimes\mathcal{W}_j)\simeq \mathrm{End}(\mathcal{V}_i)\otimes\mathrm{End}(\mathcal{W}_j)$. Now the assumption on the quaternion algebras ensures that $\mathrm{End}(\mathcal{V}_i\boxtimes\mathcal{W}_j)$ is again a quaternion algebra and hence a central division algebra over $k$. This follows from the fact that $\mathrm{End}(\mathcal{V}_i)$ (resp. $\mathrm{End}(\mathcal{W}_j)$) is by construction Brauer-equivalent to $(a,b)^{\otimes i}$ (resp. $(c,d)^{\otimes j}$) and from \cite{GS}, Theorem 1.5.5. The rest of the proof is left to the reader.
\end{proof}
\begin{rema}
\textnormal{The assumption on the quaternion algebras in Proposition 5.4 is of technically nature and ensures that any thread is a full weak exceptional collection. If we would deal with full semi-exceptional collections in each thread, this assumption can be omitted.} 
\end{rema}
\section{Rigid sheaves and exceptional vector bundles}

\begin{defi}
\textnormal{Let $X$ be a Fano variety over a field $k$. A coherent sheaf $\mathcal{F}$ is called \emph{rigid} if $\mathrm{Ext}^1(\mathcal{F},\mathcal{F})=0$.}
\end{defi}
Recall the following theorems of Kuleshov and Orlov (see \cite{KO}, Theorems 5.2 and 6.11).
\begin{thm}
An arbitrary rigid bundle on a del Pezzo surface over an algebraically closed field splits into a direct sum of exceptional bundles.
\end{thm}
\begin{thm}
On an arbitrary del Pezzo surface over an algebraically closed field each exceptional collection is part of a full exceptional collection.
\end{thm}

In \cite{RU} Rudakov formulated several conjectures. We summarize two of them in the following:
\begin{con}
\textnormal{Any rigid sheaf on a del Pezzo surface over an arbitrary field $k$ is the direct sum of weak exceptional ones and any weak exceptional sheaf can be included into a full w-exceptional collection.}
\end{con} 
Moreover, we modify the conjecture of Rudakov and formulate:
\begin{que}
\textnormal{Is any rigid vector bundle on a Fano variety over an arbitrary field $k$ the direct sum of separable exceptional vector bundles and is it possible to include any separable exceptional vector bundle into a full separable exceptional collection?}
\end{que}

To prove our main results from the introduction, we recall some facts on classical descent theory for vector bundles on proper $k$-schemes. For details see \cite{AE} and \cite{NO}.

For a vector bundle $\mathcal{E}$ on a proper $k$-scheme $X$ we set $A(\mathcal{E}):=\mathrm{End}(\mathcal{E})/\mathrm{rad}(\mathrm{End}(\mathcal{E}))$, where $\mathrm{rad}(\mathrm{End}(\mathcal{E}))$ is the Jacobson radical of the endomorphism ring. Furthermore, $Z(\mathcal{E})$ denotes the center of $A(\mathcal{E})$. The assumption that $X$ is a proper $k$-scheme ensures that vector bundles or more generally coherent sheaves enjoy a Krull--Schmidt decomposition. Thus $A(\mathcal{E})$ is a semisimple $k$-algebra. If $\mathcal{E}=\bigoplus^m_{i=1}\mathcal{E}^{\oplus d_i}_i$ is the Krull--Schmidt decomposition, then $A(\mathcal{E})$ is the product of the matrix algebras $M_{d_i}(A(\mathcal{E}_i))$. In particular, $\mathcal{E}$ is indecomposable if and only if $A(\mathcal{E})$ is a division algebra over $k$. If $k\subset L$ is a separable extension, we have $A(\mathcal{E})\otimes_k L=A(\mathcal{E}\otimes_k L)$.
\begin{lem}
Let $\mathcal{E}$ be an indecomposable vector bundle on $X$ and $k\subset L$ a normal extension containing $Z(\mathcal{E})$ and splitting $A(\mathcal{E})$. Write $m=[Z(\mathcal{E}):k]_{\mathrm{sep}}$ and $d=\mathrm{deg}_{Z(\mathcal{E})}(A(\mathcal{E}))$. Then $\mathcal{E}\otimes_k L$ has a Krull--Schmidt decomposition of the form $\bigoplus^m_{i=1}(\mathcal{E}_i)^{\oplus d}$, where $\mathcal{E}_i$ are indecomposable vector bundles on $X\otimes_k L$ with $A(\mathcal{E}_i)=L$.
\end{lem}
\begin{proof}
This is Lemma 1.1 of \cite{AE}.
\end{proof}
\begin{prop}
Let $X$ be a proper variety over $k$ and $\mathcal{F}$ and $\mathcal{G}$ coherent sheaves. If $\mathcal{F}\otimes_k \bar{k}\simeq \mathcal{G}\otimes_k \bar{k}$, then $\mathcal{F}$ is isomorphic to $\mathcal{G}$.
\end{prop}
\begin{proof}
See \cite{NO}, Proposition 3.3.
\end{proof}
\begin{rema}
\textnormal{The proof of Proposition 6.7 shows that the statement also holds for $k^s$ instead of $\bar{k}$.}
\end{rema}
\begin{prop}
Let $X$ be a proper variety over $k$. If $\mathcal{E}$ is an indecomposable $\mathrm{Gal}(k^s|k)$-invariant vector bundle on $X\otimes_k k^s$, then there exists an up to isomorphism unique indecomposable vector bundle $\mathcal{V}$ on $X$ such that $\mathcal{V}\otimes_k k^s\simeq \mathcal{E}^{\oplus m}$.  
\end{prop}
\begin{proof}
This is \cite{AE}, Proposition 3.4. We reproduce the proof as its idea will be used for the proof of Proposition 6.10 below. So let $k\subset M$ be a finite Galois extension inside of $k^s$ such that $\mathcal{E}\simeq \mathcal{N}\otimes _M k^s$ for some vector bundle $\mathcal{N}$ on $X\otimes_k M$. Then let $\pi_*\mathcal{N}$ be the sheaf on $X$ obtained by the projection $\pi:X\otimes_k M\rightarrow X$. As the $\mathrm{Gal}(k^s|k)$-conjugates of $\mathcal{N}\otimes _M k^s$ are all isomorphic to $\mathcal{E}$, we have $\pi^*\pi_*\mathcal{N}\simeq \mathcal{E}^{\oplus [M:k]}$. Applying the Krull--Schmidt Theorem we can consider a direct summand $\mathcal{M}$ of $\pi_*\mathcal{N}$. Since $\mathcal{E}$ is indecomposable, the vector bundle $\mathcal{M}$ satisfies $\mathcal{M}\otimes_k k^s\simeq \mathcal{E}^{\oplus d}$ for a suitable positive integer $d>0$. To prove the uniqueness, we assume that there is another indecomposable vector bundle $\mathcal{M}'$ satisfying $\mathcal{M}'\otimes_k k^s\simeq \mathcal{E}^{\oplus d'}$. Then $(\mathcal{M}^{\oplus d'})\otimes_k k^s\simeq (\mathcal{M}'^{\oplus d})\otimes_k k^s$, and Proposition 6.7 in combination with Remark 6.8 implies $\mathcal{M}^{\oplus d'}\simeq \mathcal{M}'^{\oplus d}$. The Krull--Schmidt Theorem yields $\mathcal{M}\simeq \mathcal{M}'$.
\end{proof}

To continue, we need the following modification of Proposition 6.9.
\begin{prop}
Let $X$ be a smooth projective variety over $k$. Let $\mathcal{E}$ be an indecomposable vector bundle on $X\otimes_k k^s$ and suppose $\{\mathcal{E}_1,...,\mathcal{E}_r\}$ is the $\mathrm{Gal}(k^s|k)$-orbit of $\mathcal{E}$. Then there is an up to isomorphism unique indecomposable vector bundle $\mathcal{F}$ on $X$ such that $\mathcal{F}\otimes_k k^s\simeq \bigoplus^r_{i=1}\mathcal{E}^{\oplus d}_i$ for a unique positive integer $d>0$.
\end{prop}
\begin{proof}
Note that the vector bundle $\mathcal{V}:=\mathcal{E}_1\oplus \mathcal{E}_2\oplus...\oplus \mathcal{E}_r$ is by assumption Galois invariant. Since $\mathcal{E}$ is indecomposable it follows that any $\mathcal{E}_i$ is indecomposable, too. To get our assertion, we proceed as in the proof of Proposition 6.9 to obtain a vector bundle $\mathcal{M}$ on $X$ such that $\mathcal{M}\otimes_k k^s\simeq \mathcal{V}^{\oplus m}=\bigoplus^r_{i=1}\mathcal{E}^{\oplus m}_i$ for a suitable positive integer $m>0$. Take any direct summand $\mathcal{W}$ of $\mathcal{M}$ and observe that $\mathcal{W}\otimes_k k^s\simeq \mathcal{V}^{\oplus r}$ for some positive integer $r\leq m$. In fact this follows from the assumption that $\{\mathcal{E}_1,...,\mathcal{E}_r\}$ is the $\mathrm{Gal}(k^s|k)$-orbit of the indecomposable bundle $\mathcal{E}$ and since $\bigoplus^r_{i=1}\mathcal{E}^{\oplus m}_i$ is the Krull--Schmidt decomposition of $\mathcal{M}\otimes_k k^s$. Now choose among the direct summands of $\mathcal{M}$ a bundle $\mathcal{F}$ with the smallest rank and denote this rank by $d$. Finally, one proceeds as in the proof of Proposition 6.9 to conclude with the Krull--Schmidt Theorem, Proposition 6.7 and Remark 6.8 that $\mathcal{F}$ is unique up to isomorphism.
\end{proof}
\begin{defi}
\textnormal{Let $X$ be a smooth projective variety over $k$. Let $\mathcal{E}$ be an indecomposable vector bundle on $X\otimes_k k^s$ and suppose $\mathbb{E}:=\{\mathcal{E}_1,...,\mathcal{E}_r\}$ is the $\mathrm{Gal}(k^s|k)$-orbit of $\mathcal{E}$. Let $d>0$ be the unique smallest positive integer for which there is an indecomposable vector bundle $\mathcal{V}$ on $X$ with $\mathcal{V}\otimes_k k^s\simeq \bigoplus^r_{i=1}\mathcal{E}^{\oplus d}_i$. We call the set $\mathbb{E}^{\oplus d}:=\{\mathcal{E}^{\oplus d}_1,...,\mathcal{E}^{\oplus d}_r\}$ the \emph{minimal descent-orbit} of $\mathbb{E}$.}
\end{defi}

The next two theorems give us a necessary and sufficient condition for Question 6.5 to be true. 


\begin{thm}
Let $X$ be a Fano variety over $k$ and $\mathcal{E}$ a rigid vector bundle on $X$. Let $\mathcal{E}\otimes_k k^s=\mathcal{A}_1\oplus...\oplus\mathcal{A}_r$ be the Krull--Schmidt decomposition on $X\otimes_k k^s$. Denote by $\mathbb{B}_i=\{\mathcal{B}_{i_{1}},...,\mathcal{B}_{i_{r_i}}\}$ the orbit of $\mathcal{A}_i$ under the action of $\mathrm{Gal}(k^s|k)$. Then $\mathcal{E}$ is the direct sum of separable exceptional vector bundles if and only if the minimal descent-orbits of all $\mathbb{B}_i$ are split semisimple exceptional blocks over $k^s$.  
\end{thm}
\begin{proof}
Since $\mathcal{E}$ is rigid on $X$, we see that $\mathcal{E}':=\mathcal{E}\otimes_k k^s$ is rigid on $X':=X\otimes_k k^s$ and $\tilde{\mathcal{E}'}=\mathcal{E}'\otimes_{k^s}\bar{k}$ on $X'\otimes_{k^s}\bar{k}$. Now consider the Krull--Schmidt decomposition of $\mathcal{E}'$ on $X'$
\begin{eqnarray*}
\mathcal{E}'=\mathcal{A}_1\oplus\mathcal{A}_2\oplus...\oplus \mathcal{A}_r.
\end{eqnarray*}
Note that all $\mathcal{A}_j$ are indecomposable vector bundles on $X'$. Proposition 3.1 of \cite{PU} implies that all $\tilde{\mathcal{A}_j}=\mathcal{A}_j\otimes_{k^s}\bar{k}$ remain indecomposable vector bundles after base change to $X'\otimes_{k^s}\bar{k}$. Now Theorem 6.2 yields that $\tilde{\mathcal{E}'}=\mathcal{E}'\otimes_{k^s}\bar{k}$ decomposes as the direct sum of exceptional vector bundles on $X'\otimes_{k^s}\bar{k}$. Let 
\begin{eqnarray}
\tilde{\mathcal{E}'}=\mathcal{G}_1\oplus\mathcal{G}_2\oplus...\oplus \mathcal{G}_s
\end{eqnarray} 
be such a decomposition. Since $\tilde{\mathcal{A}_j}$ are indecomposable, we get a second decomposition of $\tilde{\mathcal{E}'}$ into indecomposable vector bundles which is given as  
\begin{eqnarray}
\tilde{\mathcal{E}'}=\mathcal{E}'\otimes_{k^s}\bar{k}=\tilde{\mathcal{A}_1}\oplus\tilde{\mathcal{A}_2}\oplus...\oplus \tilde{\mathcal{A}_r}.
\end{eqnarray}
The Krull--Schmidt Theorem however implies $r=s$ and that decompositions (1) and (2) are up to permutation the same.
In particular we have $\mathrm{End}(\mathcal{A}_j)=k^s$. By Proposition 6.10 there are positive integers $d_i>0$ such that the vector bundles $\bigoplus^{r_i}_{j=1}\mathcal{B}^{\oplus d_i}_{i_{j}}$ descent to (indecomposable) vector bundles $\mathcal{V}_i$ on $X$. From the assumption that any $\mathbb{B}^{\oplus d_i}_i$ is a split semisimple exceptional block, we conclude $\mathrm{Ext}^l(\mathcal{V}_i,\mathcal{V}_i)=0$ for $l\neq 0$. Moreover, we have 
\begin{eqnarray*}
\mathrm{End}(\mathcal{V}_i)\otimes_k k^s &\simeq &\mathrm{End}(\bigoplus^{r_i}_{j=1}\mathcal{B}^{\oplus d_i}_{i_{j}})\\
&\simeq &M_{m_1}(k^s)\times M_{m_2}(k^s)\times...\times M_{m_t}(k^s),
\end{eqnarray*}
and hence $\mathrm{End}(\mathcal{V}_i)$ is a separable algebra over $k$. We set $d$ to be the least common multiple of all $d_i$. By the definition of $d$ there are positive integers $n_i\in\mathbb{Z}$ such that $n_i\cdot d_i=d$. We then have
\begin{eqnarray*}
\mathcal{E}^{\oplus d} &\simeq &\bigoplus^r_{i=1}(\bigoplus^{r_i}_{j=1}(\mathcal{B}^{\oplus d}_{i_{j}}))\\
&\simeq &\bigoplus^r_{i=1}(\bigoplus^{r_i}_{j=1}(\mathcal{B}^{\oplus (d_i\cdot n_i)}_{i_{j}}))\\
&\simeq &\mathcal{V}_1^{\oplus n_1}\oplus...\oplus\mathcal{V}_r^{\oplus n_r}.
\end{eqnarray*}
From the Krull--Schmidt Theorem and the fact that $\mathcal{V}_1,...,\mathcal{V}_r$ are separable-exceptional, we conclude that $\mathcal{E}$ is the direct sum of separable exceptional vector bundles on $X$.

For the other implication assume the vector bundle $\mathcal{E}$ is the direct sum of separable exceptional vector bundles, say
\begin{eqnarray*}
\mathcal{E}\simeq \mathcal{C}_1\oplus...\oplus\mathcal{C}_q.
\end{eqnarray*}
For any $\mathcal{C}_i$, let $\mathcal{C}_i=\bigoplus^{m_i}_{j=1}\mathcal{D}^{\oplus d_j}_{i_{j}}$ be its Krull--Schmidt decomposition. The $\mathcal{D}_{i_{j}}$ are indecomposable and therefore $\mathrm{End}(\mathcal{D}_{i_{j}})$ are division algebras over $k$. In particular, $\mathrm{End}(\mathcal{D}_{i_{j}})$ are separable algebras over $k$. From Lemma 6.6 we get that any $\mathcal{D}_{i_{j}}$ decomposes over $k^s$ as
\begin{eqnarray*}
\mathcal{D}_{i_{j}}\otimes_k k^s\simeq \bigoplus^{s_{i,j}}_{l=1}(\widetilde{\mathcal{D}_{i_{j}}})^{\oplus b_i}_l,
\end{eqnarray*}
with unique $b_i$ and $(\widetilde{\mathcal{D}_{i_{j}}})_l$ indecomposable. Therefore, $\mathrm{End}((\widetilde{\mathcal{D}_{i_{j}}})_l)=k^s$. The assumption $\mathrm{Ext}^r(\mathcal{C}_i,\mathcal{C}_i)=0$ for $r>0$  implies after base change to $k^s$ that $\mathrm{Ext}^r((\widetilde{\mathcal{D}_{i_{j}}})_l,(\widetilde{\mathcal{D}_{i_{j}}})_l)=0$ for $r>0$. Now since $\mathcal{C}_1,...,\mathcal{C}_q$ are separable exceptional, we see that $\mathcal{E}\otimes_k k^s$ decomposes as the direct sum of the exceptional vector bundles $(\widetilde{\mathcal{D}_{i_{j}}})_l$ on $X'$. Moreover, since the sets $\{(\widetilde{\mathcal{D}_{i_{j}}})_1,...,(\widetilde{\mathcal{D}_{i_{j}}})_{s_{i,j}}\}$ are by construction $\mathrm{Gal}(k^s|k)$-invariant, they decompose as the disjoint union of the $\mathrm{Gal}(k^s|k)$-orbits. Let us denote these orbits by $\mathbb{B}_{i_1},...,\mathbb{B}_{i_{q_i}}$. By construction, the minimal descent orbits $\mathbb{B}^{\oplus b_i}_{i_j}$, $j=1,...,q_i$, form split semisimple exceptional blocks. We see that in the Krull--Schmidt decomposition of $\mathcal{E}\otimes_k k^s$ the Galois orbits of the direct summands can be rearranged in such a way that the obtained decomposition give rise to minimal descent orbits forming split semisimple exceptional blocks. This completes the proof.
\end{proof}

\begin{thm}
Let $X$ be a Fano variety over $k$ and $\mathcal{E}$ a separable-exceptional vector bundle. Let $\mathcal{E}\otimes_k k^s=\mathcal{A}_1\oplus...\oplus\mathcal{A}_r$ the Krull--Schmidt decomposition on $X\otimes_k k^s$. Denote by $\mathbb{B}_i=\{\mathcal{B}_{i_{1}},...,\mathcal{B}_{i_{r_i}}\}$ the orbit of $\mathcal{A}_i$ under the action of $\mathrm{Gal}(k^s|k)$. Then $\mathcal{E}$ can be included into a full separable exceptional collection on $X$ if and only if the $\mathbb{B}_i$ can be included into a full exceptional collection on $X\otimes_k k^s$ consisting of Galois invariant split semisimple blocks.
\end{thm}
\begin{proof}
Assume $\mathcal{E}$ can be included into a full separable exceptional collection $\{\mathcal{E}_1,...,\mathcal{E}_m\}$. Without loss of generality $\mathcal{E}=\mathcal{E}_1$. Consider the Krull--Schmidt decompositions of the $\mathcal{E}_i\otimes_k k^s$, given by
\begin{eqnarray*}
\mathcal{E}_i\otimes_k k^s=\mathcal{F}_{i_1}\oplus \mathcal{F}_{i_2}\oplus...\oplus\mathcal{F}_{i_{s_i}}.
\end{eqnarray*}
Since all $\mathcal{E}_i$ are separable-exceptional, we obtain that all $\mathcal{F}_{i_j}$ satisfy $\mathrm{Ext}^l(\mathcal{F}_{i_j},\mathcal{F}_{i_j})=0$ for $l>0$. Moreover, $\mathrm{End}(\mathcal{F}_{i_j})=k^s$ and hence all $\mathcal{F}_{i_j}$ are exceptional vector bundles on $X':=X\otimes_k k^s$. The Galois group $\mathrm{Gal}(k^s|k)$ acts on the sets $\{\mathcal{F}_{i_1},...,\mathcal{F}_{i_{s_i}}\}$ since $\mathcal{E}_i\otimes_k k^s$ is Galois invariant. We then obtain the following block decomposition
\begin{eqnarray*}
D^b(X')=\left\langle\begin{bmatrix}
\mathcal{F}_{1_1}\\
\mathcal{F}_{1_2}\\
\vdots\\
\mathcal{F}_{1_{s_1}}\\
\end{bmatrix},
\begin{bmatrix}
\mathcal{F}_{2_1}\\
\mathcal{F}_{2_2}\\
\vdots\\
\mathcal{F}_{2_{s_2}}\\
\end{bmatrix},...,
\begin{bmatrix}
\mathcal{F}_{m_1}\\
\mathcal{F}_{m_2}\\
\vdots\\
\mathcal{F}_{m_{s_m}}\\
\end{bmatrix} \right\rangle.
\end{eqnarray*}
Note that the $\mathbb{B}_i$ are part of the first block and can therefore be included into a block decomposition with each block being Galois invariant. Obviously, all blocks occurring in the block decomposition are split semisimple by construction. 

For the other implication assume the Galois orbits $\mathbb{B}_1,...,\mathbb{B}_r$ of can be included into a full exceptional collection on $X'$ consisting of Galois invariant split-semisimple blocks. Let us denote this block decomposition by
\begin{eqnarray*}
D^b(X')=\langle \mathbb{B}_1,...,\mathbb{B}_r,\mathbb{B}_{r+1},...,\mathbb{B}_{r+m}\rangle.
\end{eqnarray*} Note that it is no restriction if we put $\mathbb{B}_1,...,\mathbb{B}_r$ at the first components of our block decomposition. We denote by $\mathcal{E}_{l_1},...,\mathcal{E}_{l_{q_l}}$ the vector bundles in the block $\mathbb{B}_l$, i.e $\mathbb{B}_l =\{\mathcal{E}_{l_1},...,\mathcal{E}_{l_{q_l}}\}$ for $l=1,...,r+m$. Since $\mathrm{End}(\bigoplus^{q_l}_{i=1}\mathcal{E}_{l_i})$ is the product of matrix algebras over $k^s$, we conclude $\mathrm{End}(\mathcal{E}_{l_i})=k^s$ and hence all $\mathcal{E}_{l_i}$ are indecomposable vector bundles on $X'$. From Proposition 6.10 we know that for the orbits $\mathbb{B}_1,...,\mathbb{B}_r$ there exist unique vector bundles $\mathcal{V}_1,...,\mathcal{V}_r$ on $X$, such that $\mathcal{V}_l\otimes_k k^s\simeq \bigoplus^{q_l}_{i=1}\mathcal{E}_{l_i}^{\oplus d_l}$ for suitable positive integers $d_l$ and $l=1,...,r$. Denote by $d$ the least common multiple of $d_1,...,d_r$. By definition there are positive integers $n_l$ such that $d=d_l\cdot n_l$. Then we have
\begin{eqnarray*}
\mathcal{E}^{\oplus d}\otimes_k k^s\simeq \mathcal{V}_1^{\oplus n_1}\oplus...\oplus \mathcal{V}_r^{\oplus n_r}.
\end{eqnarray*} 
In the same way one can show that there are vector bundles $\mathcal{R}_{r+1},...,\mathcal{R}_{r+m}$ on $X$ such that for suitable $d_j$ the bundles $\mathcal{R}^{\oplus d_j}_{r+j}$, for $j=1,...,m$, are after base change to $k^s$ isomorphic to $\mathcal{E}^{\oplus p_1}_{l_1}\oplus...\oplus\mathcal{E}^{\oplus p_{m_i}}_{l_{q_l}}$ where the $\mathcal{E}_{l_1},...,\mathcal{E}_{l_{q_l}}$ are the vector bundles occurring in the block $\mathbb{B}_{r+j}$, for $j=1,...,m$. We then get a semiorthogonal decomposition 
\begin{eqnarray*}
D^b(X) &\simeq &\langle \mathcal{E}^{\oplus d},\mathcal{R}^{\oplus d_1}_{r+1},...,\mathcal{R}^{\oplus d_r}_{r+m}\rangle\\
&\simeq & \langle \mathcal{E},\mathcal{R}_{r+1},...,\mathcal{R}_{r+m}\rangle
\end{eqnarray*}
and observe that $\mathcal{E}$ can be included into a full separable exceptional collection on $X$. 
\end{proof}

\begin{thm}
Let $X$ be a Fano variety over $k$ and $\mathcal{E}$ a rigid vector bundle. Let $\mathcal{E}\otimes_k k^s=\mathcal{A}_1\oplus...\oplus\mathcal{A}_r$ the Krull--Schmidt decomposition on $X\otimes_k k^s$. Denote by $\mathbb{B}_i=\{\mathcal{B}_{i_{1}},...,\mathcal{B}_{i_{r_i}}\}$ the orbit of $\mathcal{A}_i$ under the action of $\mathrm{Gal}(k^s|k)$. If $\#\mathbb{B}_i=1$, then $\mathcal{E}$ is the direct sum of weak exceptional sheaves. 
\end{thm}
\begin{proof}
Since $\mathcal{E}$ is rigid on $X$, we see that $\mathcal{E}':=\mathcal{E}\otimes_k k^s$ is rigid on $X':=S\otimes_k k^s$ and $\tilde{\mathcal{E}'}=\mathcal{E}'\otimes_{k^s}\bar{k}$ on $X'\otimes_{k^s}\bar{k}$. Now consider the Krull--Schmidt decomposition of $\mathcal{E}'$ on $X'$
\begin{eqnarray*}
\mathcal{E}'=\mathcal{A}_1\oplus\mathcal{A}_2\oplus...\oplus \mathcal{A}_r,
\end{eqnarray*}
where $\mathcal{A}_j$ are indecomposable vector bundles on $X'$. Proposition 3.1 of \cite{PU} implies that all $\tilde{\mathcal{A}_j}=\mathcal{A}_j\otimes_{k^s}\bar{k}$ remain indecomposable vector bundles after base change to $X'\otimes_{k^s}\bar{k}$. Now Theorem 6.2 yields that $\tilde{\mathcal{E}'}=\mathcal{E}'\otimes_{k^s}\bar{k}$ decomposes as the direct sum of exceptional vector bundles on $X'\otimes_{k^s}\bar{k}$. Let 
\begin{eqnarray}
\tilde{\mathcal{E}'}=\mathcal{G}_1\oplus\mathcal{G}_2\oplus...\oplus \mathcal{G}_s
\end{eqnarray} 
be such a decomposition. Since $\tilde{\mathcal{A}_j}$ are indecomposable, we get a second decomposition of $\tilde{\mathcal{E}'}$ into indecomposable vector bundles given as  
\begin{eqnarray}
\tilde{\mathcal{E}'}=\mathcal{E}'\otimes_{k^s}\bar{k}=\tilde{\mathcal{A}_1}\oplus\tilde{\mathcal{A}_2}\oplus...\oplus \tilde{\mathcal{A}_r}.
\end{eqnarray}
The Krull--Schmidt Theorem however implies that $r=s$ and therefore decompositions (3) and (4) are up to permutation the same.
In particular we have $\mathrm{End}(\mathcal{A}_j)=k^s$. It is also easy to see that all $\mathcal{A}_j$ are actually exceptional vector bundles on $X'$. 
Moreover, from the assumption $\#\mathbb{B}_i=1$ we obtain that all $\mathcal{A}_j$ are $\mathrm{Gal}(k^s|k)$-invariant. So we can apply Proposition 6.9 to get indecomposable vector bundles $\mathcal{V}_j$ on $X$ such that $(\mathcal{V}_j)\otimes_k k^s\simeq \mathcal{A}^{\oplus m_j}_j$. Now observe that  
\begin{eqnarray*}
\mathrm{End}(\mathcal{V}_j)\otimes_k k^s &\simeq &\mathrm{End}(\mathcal{A}^{\oplus m_j}_j)\\
&\simeq &M_{m_j}(k^s),
\end{eqnarray*}
and hence $\mathrm{End}(\mathcal{V}_j)$ is a central simple algebra over $k$. We recall that $\mathcal{V}_j$ is indecomposable so that $\mathrm{End}(\mathcal{V}_j)$ is indeed a central division algebra. Moreover, it follows easily from $(\mathcal{V}_j)\otimes_k k^s\simeq \mathcal{A}^{\oplus m_j}_j$ that $\mathrm{Ext}^l(\mathcal{V}_j,\mathcal{V}_j)=0$ for $l>0$. Thus all $\mathcal{V}_j$ are weak exceptional vector bundles. Now consider the bundles $\mathcal{V}_j$ for $j=1,...,r$ and set $d:=\mathrm{lcm}(m_1,m_2,...,m_r)$ to be the least common multiple. By the definition of the least common multiple there are positive integers $n_j\in \mathbb{Z}$ such that $n_j\cdot m_j=d$. We now consider the vector bundle $\mathcal{E}^{\oplus d}$ and find
\begin{eqnarray*}
(\mathcal{E}^{\oplus d})\otimes_k k^{s}\simeq (\bigoplus_{j=1}^r\mathcal{A}_j)^{\oplus d}\simeq \bigoplus_{i=j}^r\mathcal{A}_j^{\oplus m_j\cdot n_j}.
\end{eqnarray*} Since the vector bundles $\mathcal{A}_j^{\oplus m_j}$ descent to $\mathcal{V}_j$, we obtain
\begin{eqnarray*}
(\mathcal{E}^{\oplus d})\otimes_k k^s\simeq (\bigoplus_{j=1}^r\mathcal{V}_{j}^{\oplus n_j})\otimes_k k^s.
\end{eqnarray*} From Proposition 6.7 we conclude 
\begin{eqnarray*}
\mathcal{E}^{\oplus d}\simeq \bigoplus_{j=1}^r\mathcal{V}_{j}^{\oplus n_j}.
\end{eqnarray*} The Krull--Schmidt Theorem now implies that $\mathcal{E}$ has to be the direct sum of some of the $\mathcal{V}_j$. Hence it is the direct sum of weak exceptional sheaves. This completes the proof.
\end{proof}
\begin{exam}
\textnormal{Let $C$ be a Brauer--Severi curve over $k$ and $\mathcal{E}$ a rigid vector bundle. After base change to $k^s$ the bundle $\mathcal{E}\otimes_k k^s$ on $C\otimes_k k^s\simeq \mathbb{P}^1_{k^s}$ is according to a theorem of Grothendieck the direct sum of $\mathcal{O}_{\mathbb{P}^1_{k^s}}(i)$. Since any $\mathcal{O}_{\mathbb{P}^1_{k^s}}(i)$ is $\mathrm{Gal}(k^s|k)$ invariant, Theorem 6.13 implies that $\mathcal{E}$ is the direct sum of w-exceptional vector bundles. Note that this fact also follows from the classification of vector bundles on $C$. In \cite{NO} it is proved that any vector bundle $\mathcal{E}$ is the direct sum of indecomposable $AS$-bundles. The indecomposable $AS$-bundles however are weak exceptional vector bundles so that in fact any bundle on $C$ is the direct sum of weak exceptional ones.}
\end{exam}

{\small MATHEMATISCHES INSTITUT, HEINRICH--HEINE--UNIVERSIT\"AT 40225 D\"USSELDORF, GERMANY}\\
E-mail adress: novakovic@math.uni-duesseldorf.de

\end{document}